\documentclass[letterpaper, 10 pt, conference]{ieeeconf}  

\IEEEoverridecommandlockouts                              
\usepackage{amsmath,amsfonts}
\usepackage{tikz}
\usepackage{booktabs}
\usetikzlibrary{automata,arrows,positioning,calc} 
\usepackage[font=footnotesize]{caption}
\usepackage{cite}
\usepackage{color}

\newtheorem{theorem}{Theorem}

\newtheorem{remark}{Remark}

\newtheorem{proposition}{Proposition}

\newcommand{\numroads}{n}
\newcommand{\setroads}{{[}\numroads{]}}
\newcommand{\idxroad}{i}
\newcommand{\idxroadalt}{{i'}}
\newcommand{\idxmixedrd}{m}

\newcommand{\numhumroads}{n^\text{h}}
\newcommand{\sethumroads}{{[}\numhumroads{]}}
\newcommand{\numautroads}{n^\text{a}}
\newcommand{\setautroads}{{[}\numautroads{]}}

\newcommand{\flowhum}{f^\text{h}}
\newcommand{\flowaut}{f^\text{a}}
\newcommand{\flowopt}{f^*}
\newcommand{\flowhumopt}{f^\text{*h}}
\newcommand{\flowautopt}{f^\text{*a}}

\newcommand{\alteqflow}{\tilde{f}}
\newcommand{\alteqhumflow}{\tilde{f}^\text{h}}
\newcommand{\alteqautflow}{\tilde{f}^\text{a}}

\newcommand{\alteqhumflowtwo}{\tilde{g}^\text{h}}

\newcommand{\altopthumflowtwo}{g^\text{*h}}

\newcommand{\demand}{\bar{f}}
\newcommand{\demandhum}{\bar{f}^\text{h}}
\newcommand{\demandaut}{\bar{f}^\text{a}}

\newcommand{\autlev}{\alpha}

\newcommand{\tollhum}{\tau^\text{h}}
\newcommand{\tollaut}{\tau^\text{a}}
\newcommand{\largenum}{P}

\newcommand{\capacity}{q}
\newcommand{\caphum}{m}
\newcommand{\capaut}{M}
\newcommand{\rdlen}{d}
\newcommand{\latency}{\ell}

\newcommand{\fflat}{t}
\newcommand{\bprparam}{\rho}	
\newcommand{\bprexpo}{\sigma}
\newcommand{\scaleparam}{a}
\newcommand{\asymmetry}{k}

\newcommand{\costhum}{c^\text{h}}
\newcommand{\costaut}{c^\text{a}}

\newcommand{\eqcosthum}{\hat{c}^\text{h}}
\newcommand{\eqcostaut}{\hat{c}^\text{a}}

\newcommand{\socialcost}{C}

\title{\LARGE \bf
Optimal Tolling \\ for Heterogeneous Traffic Networks  with Mixed Autonomy
}

\author{Daniel A. Lazar$^{1}$, Samuel Coogan$^{2}$, and Ramtin Pedarsani$^{1}$
	\thanks{$^{1}$Department of Electrical and Computer Engineering, 
		University of California, Santa Barbara}%
	\thanks{$^{2}$School of Electrical and Computer Engineering and School of Civil and Environmental Engineering, Georgia Institute of Technology}%
	\thanks{{\tt\small dlazar@ece.ucsb.edu, sam.coogan@gatech.edu, ramtin@ece.ucsb.edu}}
}

\begin{document}

\maketitle
\thispagestyle{empty}
\pagestyle{empty}

\begin{abstract}

When people pick routes to minimize their travel time, the total experienced delay, or social cost, may be significantly greater than if people followed routes assigned to them by a social planner. This effect is accentuated when human drivers share roads with autonomous vehicles. When routed optimally, autonomous vehicles can make traffic networks more efficient, but when acting selfishly, the introduction of autonomous vehicles can actually \emph{worsen} congestion. We seek to mitigate this effect by influencing routing choices via tolling. We consider a network of parallel roads with affine latency functions that are \emph{heterogeneous}, meaning that the increase in capacity due to to the presence of autonomous vehicles may vary from road to road. We show that if human drivers and autonomous users have the same tolls, the social cost may be arbitrarily worse than when optimally routed. We then prove qualities of the optimal routing and use them to design tolls that are guaranteed to minimize social cost at equilibrium. To the best of our knowledge, this is the first tolling scheme that yields a unique socially optimal equilibrium for parallel heterogeneous network with affine latency functions.

\end{abstract}

\section{INTRODUCTION}
\label{sct:intro}
Road congestion in the United States alone costs billions of dollars in wasted time and fuel, even neglecting the associated environmental degradation and negative health impacts \cite{schrank2015}. Moreover, this cost is projected to only increase. One potential way to mitigate this is through the introduction of autonomous vehicles, which can increase throughput on urban roads by a factor of two or three \cite{lioris2017platoons}. However, for the foreseeable future, autonomous vehicles will share roads with human drivers, rendering the mobility benefits less clear. 

Moreover, it is important to consider that human drivers, and likely the users of autonomous vehicles, will make selfish decisions -- they will pick routes that minimizes their travel time without considering the effect of their choice on the overall traffic congestion. This leads to a gap between the overall traffic delay in selfish equilibria and the minimum possible overall delay, which would occur if vehicles followed directions from some benevolent social planner \cite{pigou1932economics}. The gap between these costs can be bounded with a bound that depends on the relationship between road delay and the flow of vehicles on that road. However, this gap can be much greater in networks with mixed autonomy than in networks with a single vehicle type, and the gap may even be unbounded \cite{lazar2018poa}.

In a similar phenomenon, converting some fraction of vehicles to be autonomous can make equilibria \emph{worse}, even though the presence of these vehicles \emph{increases} the capacity of roads \cite{mehr2018can}. This effect is related to Braess' Paradox, in which adding a road to a road network may increase the equilibrium delay. These phenomena further complicate how congestion in traffic networks will change as autonomous vehicle manufacturers release their cars onto public roads. 

To address this, we wish to somehow influence how people choose their routes, and a natural mechanism to consider is tolling. Beckmann \emph{et al.} \cite{beckmann1956studies} found optimal tolls when there is a single type of car on the road and others extend this to when there are multiple vehicle types \cite{dafermos1972multiclass_user}. However, the setting of mixed autonomy violates key assumptions made in these classic works; consequently, optimality is no longer guaranteed. Because of this, we investigate how to guarantee socially optimal routing via tolling. We consider the case of parallel roads and, motivated by previously established capacity and delay models, we consider road delay functions that are affine with respect to the flow of each type of vehicle. We establish theoretical results about the optimal routing and use these results to design an optimal tolling scheme.

Our contributions are as follows:
\begin{itemize}
	\item We bound the increase in the overall latency when autonomous vehicles are introduced to road networks,
	\item we show that optimal undifferentiated tolls yield equilibria that can be arbitrarily worse than the social optima, and
	\item we prove qualities of the optimal routing for vehicles on shared roads and use these results to design tolls that yield a unique, socially optimal equilibrium. \\
\end{itemize}

\noindent \textbf{Previous Work. } Many works have shown how autonomous vehicles can decrease congestion, either via platooning \cite{lioris2017platoons,askari2017effect}, by dissipating shockwaves in congested vehicle flow \cite{stern2018dissipation, wu2018stabilizing}, or by managing merging vehicles at bottlenecks \cite{li2018modeling,jang2018simulation}. Other works characterize the capacity of a road as a function of the fraction of vehicles on a road that are autonomous \cite{lazar2017routing, askari2017effect}. This allows the formulation of a congestion game with mixed autonomy \cite{lazar2018routing}. Many works have studied aspects of congestion games, such as how to optimally route vehicles of differing types \cite{dafermos1969traffic_general}, as well as bound the Price of Anarchy (PoA), the maximum ratio between total latency under selfish routing to that under optimal routing \cite{koutsoupias:1999fs,roughgarden2002bad,perakis2007price, correa2008geometric}. The PoA has been bounded in the case of mixed autonomy as well \cite{lazar2018poa, lazar2018routing}. Properties of selfish routing were first introduced by Wardrop \cite{wardrop1900some} and have been expanded in many subsequent works \cite{smith1979existence, depalma1998optimization, correa2011wardrop}.

In the case of a single vehicle type, \emph{marginal cost tolls} yield an essentially unique equilibrium that minimizes total latency \cite{beckmann1956studies}. Under certain conditions, the same can be said when there are multiple vehicle types on a road \cite{dafermos1972multiclass_user}. However, these conditions are violated when the two vehicle types are capable of maintaining different headways. Specifically, \cite{dafermos1972multiclass_user} assumes that the Jacobians of the latency functions are positive definite; the latency function derived in \eqref{eq:latency} in Section~\ref{sct:model} violate this assumption.

This is addressed in \cite{mehr2019pricing}; there the authors show by example that in multicommodity networks (networks with multiple source-destination pairs), if autonomous vehicles and human drivers experience the same tolls, it is not always possible to create an equilibrium that minimizes social cost. However, they have positive results for a \emph{homogeneous} network; a network in which all roads see the same multiplicative increase in capacity from the presence of autonomous vehicles. In this case, they show that through proper tolling, all equilibria will minimize social cost. The tolling used to achieve this is \emph{marginal cost tolling}, in which each user pays the value of the marginal increase in delay they cause for all others who would be using that road in the optimal routing. This means that human drivers and autonomous users pay different amounts for traveling on a road, since human drivers will generally contribute more to congestion. 

We are interested in the case in which the network is not homogeneous, and different roads will see differing benefits from autonomy. This is relevant as platooning may yield greater capacity increases on highways than on urban roads. For heterogeneous networks and affine cost function, we develop the first tolling scheme that provably achieves a unique socially optimal equilibrium. 

On a broader level, we are motivated by the results of \cite{mehr2018can}, which show that converting some demand from human-driven vehicles to autonomous vehicles can counterintuitively worsen aggregate delay in equilibrium. In \cite{mehr2019will} the authors bound the effect of this phenomena in homogeneous networks to be no greater than the PoA for the class of cost functions with a single vehicle type. We extend this to the heterogeneous case in Section~\ref{sct:tolling}. Before doing so, we describe the model for our network and the relationship between vehicle flow and travel delay on the roads.

\section{MODEL}
\label{sct:model}

In this section we specify the structure of the road network, road latency functions, and the tolls considered, as well as characteristics of equilibria. We consider $\numroads$ parallel roads and use $\setroads = \{ 1,2,\ldots,n\}$ to denote the set of roads. We use $\flowhum_\idxroad$ and $\flowaut_\idxroad$ respectively to denote the human-driven and autonomous flow on road $\idxroad$; the vectors $\flowhum,\flowaut \in \mathbb{R}^\numroads_{\ge 0}$ denote the flow on all roads. We use the term \emph{routing} to refer to a flow pair $(\flowhum,\flowaut)$.

We consider nonatomic flow, meaning each user controls an infinitesimally small unit of the flow, and does not individually change the travel delay on a road. We use $\demandhum$ and $\demandaut$ to denote the flow demand of human-driven and autonomous vehicles, respectively. This demand is considered nonelastic, meaning that the demand is constant and independent of the road latencies.

Each road has an associated \emph{delay function} or \emph{latency function}, $\latency_\idxroad(\flowhum_\idxroad,\flowaut_\idxroad):\mathbb{R}^2_{\ge 0} \rightarrow \mathbb{R}_{\ge 0}$. A traveler who chooses a road experiences a cost that is the sum of the latency and the toll they pay to travel on the road; we assume that all users have the same sensitivity to tolls. Human drivers and autonomous users may have different tolls, respectively denoted $\tollhum_\idxroad$ and $\tollaut_\idxroad$. The \emph{experienced costs} for the vehicle types are then
\begin{align*}
	\costhum_\idxroad(\flowhum_\idxroad,\flowaut_\idxroad) &= \latency_\idxroad(\flowhum_\idxroad,\flowaut_\idxroad) + \tollhum_\idxroad \; , \\
	\costaut_\idxroad(\flowhum_\idxroad,\flowaut_\idxroad) &= \latency_\idxroad(\flowhum_\idxroad,\flowaut_\idxroad) + \tollaut_\idxroad \; .
\end{align*}

The \emph{social cost} is the total negative consequence of a given traffic pattern. We consider tolls that are recirculated back into the public coffers, so the only harm incurred to society is the latency experienced by the network users. The social cost is then as follows:
\begin{equation}\label{eq:socialcost}
	\socialcost(\flowhum,\flowaut) = \sum_{\idxroad \in \setroads}(\flowhum_\idxroad+\flowaut_\idxroad)\latency_\idxroad(\flowhum_\idxroad,\flowaut_\idxroad) \; .
\end{equation}

A road's latency function depends on the capacity of the road, which in turn depends on the autonomy level $\autlev_\idxroad$, the fraction of vehicles on road $\idxroad$ that are autonomous. The road capacity is the maximum vehicle flow that can travel on that road. Let autonomous vehicles occupy $\capaut^{-1}_\idxroad$ meters, including headway, when traveling at nominal velocity on road $\idxroad$, and let human-driven vehicles occupy $\caphum^{-1}_\idxroad$ meters. We denote the road length as $\rdlen_\idxroad$ and the free-flow velocity as $v_\idxroad$. Then, as in \cite{lazar2017routing, askari2017effect}, the road capacity is as follows\footnote{This model assumes either that autonomous vehicles do not rely on vehicle-to-vehicle communication and maintain the same headways behind the vehicle that they follow regardless of vehicle types, or that the autonomous vehicles can rearrange themselves to form one large platoon, as in \cite{lazar2018maximizing}.}.

\begin{equation*}
	\capacity_\idxroad(\autlev_\idxroad) = v_\idxroad \rdlen_\idxroad/(\autlev_\idxroad \capaut^{-1}_\idxroad + (1-\autlev_\idxroad)\caphum^{-1}_\idxroad) \; .
\end{equation*}

We use this in conjunction with the well known Bureau of Public Roads road latency model \cite{bureau1964manual,branston1976link}, as in \cite{lazar2018routing,mehr2018can}. This yields the following road latency function.
\begin{align}\label{eq:latencygeneral}
	\latency_\idxroad(\flowhum_\idxroad,\flowaut_\idxroad) = \fflat_\idxroad\left(1+\bprparam_\idxroad\left(\frac{\flowhum_\idxroad}{\caphum_\idxroad} + \frac{\flowaut_\idxroad}{\capaut_\idxroad}\right)^{\bprexpo_\idxroad}\right) \; .
\end{align}

In this paper, except for when bounding the \emph{Price of Autonomy}, we consider $\bprexpo_\idxroad = 1$ $\forall \idxroad \in \setroads$. Then we combine the parameters, with $\scaleparam_\idxroad = \fflat_\idxroad \bprparam_\idxroad/\capaut_\idxroad$ and $\asymmetry_\idxroad = \capaut_\idxroad/\caphum_\idxroad$. The latency function is then affine with respect to the vehicle flows,
\begin{equation}\label{eq:latency}
	\latency_\idxroad(\flowhum_\idxroad,\flowaut_\idxroad) = \asymmetry_\idxroad \scaleparam_\idxroad \flowhum_\idxroad + \scaleparam_\idxroad \flowaut_\idxroad + \fflat_\idxroad \; .
\end{equation}

\noindent \textbf{Equilibria. } We are concerned with characterizing Wardrop Equilibria of a traffic network, meaning situations in which no user has incentive to change their strategy \cite{wardrop1900some}. We treat equilibria as reasonable predictions of user behavior. 

We consider both human drivers and autonomous users to be selfish. This means that in Wardrop Equilibrium, if there is positive human-driven flow on road $\idxroad$, this implies that 
$$
\costhum_\idxroad(\flowhum_\idxroad,\flowaut_\idxroad) \le \costhum_\idxroadalt(\flowhum_\idxroadalt,\flowaut_\idxroadalt), ~~\forall \idxroadalt \in \setroads,
$$ 
and similarly, for autonomous flow, $\flowaut_\idxroad > 0$ implies 
$$
\costaut_\idxroad(\flowhum_\idxroad,\flowaut_\idxroad) \le \costaut_\idxroadalt(\flowhum_\idxroadalt,\flowaut_\idxroadalt), ~~\forall \idxroadalt \in \setroads.
$$ 
Since we consider parallel roads, all users of the same type will experience the same cost.

As mentioned earlier, equilibria can often incur far greater social cost than the optimal routing. Because of this, our goal is to design tolls such that the only equilibrium that exists minimizes the social cost.

\section{EFFICIENCY AND TOLLING}
\label{sct:tolling}
To begin our discussion of tolling, we expand previous results to bound how much worse equilibria can be when some vehicles are autonomous as compared to when all vehicles are human driven. Following this, we show that tolling cannot help in any significant way when we are forced to toll humans and autonomous vehicles identically. We then develop properties of optimal routing and provide a method for calculating optimal tolls. \\

\noindent \textbf{Bounding the Price of Autonomy. }As mentioned earlier, \cite{mehr2019will} shows that converting some vehicles to be autonomous may worsen the aggregate delay in equilibrium. This work bounds this effect when a network is \emph{homogeneous}, meaning that there the multiplicative increase in capacity due to autonomy is uniform on all roads, \emph{i.e.} $\capaut_\idxroad/\caphum_\idxroad = \capaut_\idxroadalt/\caphum_\idxroadalt$ $\forall \idxroad, \idxroadalt \in \setroads$. In particular, \cite{mehr2019will} shows that the delay will not increase by a factor more than the Price of Anarchy for that class of cost functions with a single vehicle type. However, if the roads have varying characteristics such as speed limits and separation from pedestrian traffic, the network will not be homogeneous and this bound will not hold. Motivated by this, we bound this quantity for general (not necessarily parallel) heterogeneous networks.

First we define the maximum degree of asymmetry
\begin{equation*}
	\asymmetry = \max_{\idxroad \in \setroads} \capaut_\idxroad/\caphum_\idxroad
\end{equation*}
and the maximum polynomial degree
\begin{equation*}
	\bprexpo = \max_{\idxroad \in \setroads}\bprexpo_\idxroad \;  .
\end{equation*}
Let 
\begin{equation*}
	\xi(\bprexpo) = \bprexpo(\bprexpo+1)^{-(\bprexpo+1)/ \bprexpo} \;  .
\end{equation*}
\begin{proposition}
	Consider a general network with social cost as defined in \eqref{eq:socialcost} and road latency functions of the form \eqref{eq:latencygeneral}, and let $\capaut_\idxroad/\caphum_\idxroad\ge 1$ $\forall \idxroad \in \setroads$. Let $(\alteqhumflowtwo,0)$ be a Wardrop Equilibrium routing with human flow demand $\demand$ and zero autonomous flow demand. Let $(\alteqhumflow,\alteqautflow)$ be an equilibrium routing with human flow demand $\demandhum$ and autonomous flow demand $\demand - \demandhum$, where $\demandhum \in {[}0,\demand{]}$. Then,
	\begin{align*}
	\socialcost(\alteqhumflow,\alteqautflow) \le \frac{\asymmetry^\bprexpo}{1-\xi(\bprexpo)}\socialcost(\alteqhumflowtwo,0)  \; .
	\end{align*}
\end{proposition}
\begin{proof}
	Let $(\altopthumflowtwo,0)$ denote the socially optimal routing for human flow demand $\demand$ and zero autonomous demand, and let $(\flowhumopt,\flowautopt)$ be a socially optimal routing for human flow demand $\demandhum$ and autonomous flow demand $\demand - \demandhum$, where $\demandhum \in {[}0,\demand{]}$.

	The results in \cite{lazar2018routing} imply that
	\begin{align*}
		\socialcost(\alteqhumflow,\alteqautflow) \le \frac{\asymmetry^\bprexpo}{1-\xi(\bprexpo)}\socialcost(\flowhumopt,\flowautopt) \; .
	\end{align*}
	
	Then,
	\begin{align*}
		\socialcost(\flowhumopt,\flowautopt) &\le \socialcost(\altopthumflowtwo,0) \\
		& \le \socialcost(\alteqhumflowtwo,0) \; ,
	\end{align*}
	due to the assumption that $\capaut_\idxroad/\caphum_\idxroad\ge 1$ $\forall \idxroad \in \setroads$ and the definition of optimal flow.
	
	Together, these imply the proposition.
\end{proof}

Note that the assumption that $\capaut_\idxroad/\caphum_\idxroad\ge 1$ $\forall \idxroad \in \setroads$ is required for this proposition, but is not required for the subsequent theoretical results.

This bound for the heterogeneous case is equal to the Price of Anarchy bound for mixed autonomy in \cite{lazar2018routing}, Theorem~1. Though the increase in inefficiency is bounded with respect to $\asymmetry$ and $\bprexpo$, it grows with these parameters. With this motivation, we look to mitigate this inefficiency through tolling. \\

\noindent \textbf{Undifferentiated tolls. }First we consider whether we can enforce optimal routing with undifferentiated tolls, meaning a tolling scheme in which human drivers and autonomous vehicles pay the same toll for a road. Can such a tolling scheme yield a unique socially-optimal equilibrium?

Previous work has answered this question negatively -- in \cite{mehr2019pricing}, the authors show an example of a network with multiple source destination pairs in which undifferentiated tolls fail to minimize social cost in equilibrium. In this section we extend these results and show a simple two-road network in which undifferentiated tolls fail. We further show that the best undifferentiated tolling scheme in this network yields an equilibrium with social cost that can be arbitrarily worse than the social cost under optimal routing.

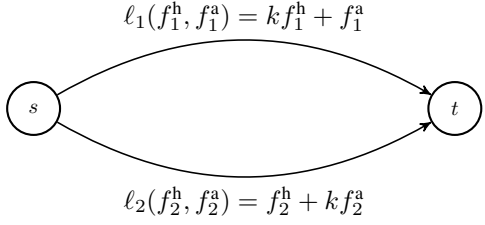
\begin{figure}
	\centering
	\begin{tikzpicture}[->, >=stealth', auto, semithick, node distance=7cm]
	\tikzstyle{every state}=[fill=white,draw=black,thick,text=black,scale=0.8]
	\node[state]    (0)               {$s$};
	\node[state]    (1)[right of=0]   {$t$};
	\path
	(0) edge[bend left]		node{$\latency_1(\flowhum_1, \flowaut_1) = \asymmetry \flowhum_1 + \flowaut_1$}     (1)
	(0) edge[bend right]	node[below]{$\latency_2(\flowhum_2,\flowaut_2) = \flowhum_2 + \asymmetry \flowaut_2$}     (1);
	\end{tikzpicture}
	\caption{Example of the futility of undifferentiated tolling in a simple network. Consider one unit of flow demand for each human-driven and autonomous vehicles -- equilibrium under the best undifferentiated toll may be arbitrarily worse than the socially optimal routing.}
	\label{fig:undiffentiatedtolls}
\end{figure}

Consider the network in Fig.~\ref{fig:undiffentiatedtolls}, with one unit of human-driven flow demand and one unit of autonomous flow demand. Let $\asymmetry \ge 1$. The socially-optimal routing has social cost $2$, with all autonomous flow on road $1$ and all human-driven flow on road $2$. The worst-case equilibrium has this routing reversed for a social cost of $2\asymmetry$.

Without loss of generality, we can consider a toll on just one of the roads, since only the difference between the tolls on the two roads will affect the equilibrium. This example is symmetric, so without loss of generality let the top road be the road with a positive toll. In the resulting worst-case equilibrium, the top road has some of the human-driven flow and the bottom road has the remainder of the human-driven flow and all the autonomous flow. To investigate how well the best toll can do, we derive the following.
\begin{align*}
	&\min_{\flowhum_1 \in {[}0,1{]}} \flowhum_1\latency_1(\flowhum_1,0) + (1-\flowhum_1+1)\latency_2(1-\flowhum_1,1) \\ 
	& = \min_{\flowhum_1 \in {[}0,1{]}} \asymmetry (\flowhum_1)^2 + (1-\flowhum_1+1)(1-\flowhum_1+\asymmetry) \\
	&= \frac{7\asymmetry + 3}{4}-\frac{1}{\asymmetry+1} < 2\asymmetry \; .
\end{align*}
We see that the toll decreases our cost from that of the worst-case equilibrium. However, the worst-case equilibrium cost increases linearly with $\asymmetry$, as does the worst-case equilibrium cost when using optimal undifferentiated tolls, while the socially optimal cost is constant. Therefore, optimal undifferentiated tolling can yield unboundedly worse social cost than the socially optimal routing in mixed autonomy, even in a two-road network with affine latency functions. Because of this, we turn our attention to the case in which we can leverage different tolls to the two vehicle types. \\

\noindent \textbf{Differentiated tolls. }What if we allow different tolls based on the type of vehicle traveling on a road? In order to develop this, we first present a result regarding the optimal routing of our vehicle flows, which will be useful in considering how to levy tolls

\begin{theorem}\label{thm:separation}
	Consider a network of parallel roads with latency functions of the form \eqref{eq:latency} and assume that $\asymmetry_\idxroad=1$ for at most one road. Then, any routing which minimizes social cost will have at most one road shared between human-driven and autonomous vehicles. 
\end{theorem}
\begin{proof}
	We prove this by contradiction. Assume that an optimal routing $\flowopt$ has positive human-driven and autonomous flow on both roads $i$ and $j$ $\in \setroads$, \emph{i.e.} $\flowhumopt_i>0,\flowautopt_i>0$ and $\flowhumopt_j>0,\flowautopt_j>0$, where $i \neq j$. We fix the sum of the flow of each type on those two roads, denoting $\demandhum_{ij} = \flowhumopt_i + \flowhumopt_j$ and $\demandaut_{ij} = \flowautopt_i + \flowautopt_j$. To show the contradiction, it is sufficient to show that there is a routing with lower social cost, with the same fixed total flow on the two roads.
	
	The cost of the flow on these two roads, as a function of $\flowhum_i$ and $\flowaut_i$ and parameterized by the demands, is as follows. 
	\begin{align*}
	&J(\flowhum_i,\flowaut_i,\demandhum_{ij},\demandaut_{ij}) = (\asymmetry_i \scaleparam_i \flowhum_i + \scaleparam_i \flowaut_i + \fflat_i)(\flowhum_i + \flowaut_i) \; + \\
	& \; (\asymmetry_j \scaleparam_j (\demandhum_{ij}-\flowhum_i) + \scaleparam_j (\demandaut_{ij} - \flowaut_i) + \fflat_j)(\demandhum_{ij} - \flowhum_i + \demandaut_{ij} - \flowaut_i) \: .
	\end{align*}
	The fact that $0 < \flowhumopt_i < \demandhum_{ij}$ and $0 < \flowautopt_i < \demandaut_{ij}$ implies that there is a minimum in the feasible set, implying that the Hessian of the cost function with respect to $\flowhumopt_i$ and $\flowautopt_i$ has positive eigenvalues at some point in the feasible set. However, the Hessian is as follows:
	\begin{equation*}
	\begin{bmatrix}
	2 \scaleparam_i \asymmetry_i + 2 \scaleparam_j \asymmetry_j & (\asymmetry_i+1)\scaleparam_i + (\asymmetry_j + 1)\scaleparam_j \\
	(\asymmetry_i+1)\scaleparam_i + (\asymmetry_j + 1)\scaleparam_j & 2 \scaleparam_i + 2 \scaleparam_j
	\end{bmatrix} \; ,
	\end{equation*}
	which has determinant $-(\scaleparam_i(\asymmetry_i-1) + \scaleparam_j(\asymmetry_j-1))^2$, which is negative if not both $\asymmetry_i$ and $\asymmetry_j$ equal one. This matrix has both positive and negative eigenvalues, implying that no local minimum exists that is not on the boundary, contradicting the premise. Therefore, no two roads can have both positive human-driven flow and autonomous flow.
\end{proof}

Note that we haven't shown that the optimal solution is unique, rather we've established a property of any optimal solution. To assist with our analysis, for a given optimal routing $\flowopt$ we denote the set of roads without autonomous vehicles as $\sethumroads$ and the set of roads without human-driven vehicles as $\setautroads$. Some of the roads may have no flow in $\flowopt$; these roads will be in both sets.

We've shown that to minimize social cost we want to separate the human-driven and autonomous flow as much as possible. By exploiting this property, we will derive an optimal tolling scheme. Conceptually, by using this property and controlling which roads each vehicle type can use, we can rule out the possible existence of suboptimal equilibria. Since we can use differentiated tolls, we influence which roads each vehicle type chooses by leveraging large tolls to keep them off of roads on which they don't travel in the chosen optimal routing. We use tolls of the form of Dafermos' path tolls \cite{dafermos1972multiclass_user} for the remaining roads. We formalize this in the following theorem.

\begin{theorem}\label{thm:tolls}
	Consider a network of parallel roads with latency functions of the form \eqref{eq:latency} and assume that $\asymmetry_\idxroad=1$ for at most one road. Further, assume that latency is strictly increasing with vehicle flow, \emph{i.e.} $\asymmetry_\idxroad > 0$ and $\scaleparam_\idxroad > 0$ $\forall \idxroad \in \setroads$. Solve for a socially optimal routing $\flowopt$, and use $\sethumroads$ to denote the set of roads without autonomous vehicles and use $\setautroads$ to denote the set of roads without human-driven vehicles in $\flowopt$. Then levy the following tolls. Let
		\begin{equation}\label{eq:tolls}
			\begin{aligned}
			& \tollhum_\idxroad =
			& & \largenum 
			& & \text{if} \; \idxroad \in \setautroads \\
			& 
			& & \mu- \latency_\idxroad(\flowhumopt_\idxroad,\flowautopt_\idxroad)
			& & \text{otherwise} \\
			& \tollaut_\idxroad = 
			& & \largenum
			& & \text{if} \; \idxroad \in \sethumroads \\
			&
			& & \mu-\latency_\idxroad(\flowhumopt_\idxroad,\flowautopt_\idxroad)
			& & \text{otherwise} \; ,
			\end{aligned}
		\end{equation}
		for arbitrary $\mu$ and sufficiently large $\largenum$. Then, all resulting equilibria will have the same social cost, which is equal to that of the socially optimal routing $\flowopt$.
\end{theorem}
\begin{proof}
	First, Theorem~\ref{thm:separation} guarantees that a socially optimal routing will have at most one mixed road. Now we must prove 1) that $\flowopt$ is an equilibrium with the new tolls and 2) that any equilibrium will have the same social cost as $\flowopt$. The first follows directly from the construction of the toll and the definition of an equilibrium. When following routing $\flowopt$, all users will experience cost $\mu$ and any other option would have cost at least $\mu$, therefore satisfying conditions for equilibrium.
	
	Now we must prove essential uniqueness for this equilibrium, meaning that any other equilibrium will have the same social cost as $\flowopt$. First note that for $\largenum$ sufficiently large, we will never have an equilibrium with human-driven flow on roads in $\setautroads$ and similarly, there will never be autonomous vehicles on roads $\sethumroads$. Now, first consider the case in which there is no mixed road. In this case we can split the network in two and consider each part separately -- Dafermos \cite{dafermos1972multiclass_user} (Proposition 3.2) ensures essential uniqueness on each part.
	
	Next, consider the event that a mixed road exists, and use $\idxmixedrd$ to denote its index. Now if we can fix the flow on $\idxmixedrd$ to be $(\flowhumopt_\idxmixedrd, \flowautopt_\idxmixedrd)$, we can again split the network in two and consider each separately. But can we guarantee that road $\idxmixedrd$ will have flow $(\flowhumopt_\idxmixedrd, \flowautopt_\idxmixedrd)$?
	
	Let us denote the total human-driven and autonomous vehicle flow as $\demandhum$ and $\demandaut$, respectively. We split the flow of each type into two: the flow on the mixed road, $\flowhum_\idxmixedrd$ and $\flowaut_\idxmixedrd$, and the remaining flow, $\demandhum-\flowhum_\idxmixedrd$ and $\demandaut-\flowaut_\idxmixedrd$. There exists a large enough $\largenum$ such that $\flowaut_\idxroad = 0$ $\forall \idxroad \in \sethumroads$ and $\flowhum_\idxroad=0$ $\forall \idxroad \in \setautroads$. As established in Section~\ref{sct:model}, all users of the same type will experience the same cost. Accordingly, all users on the roads $\sethumroads$ will experience the same cost which is increasing with respect to the flow demand -- we use $\eqcosthum:\mathbb{R}_{\ge 0} \rightarrow \mathbb{R}_{\ge 0}$ to denote the cost experienced by the users of roads $\sethumroads$ as a function of the total flow on those roads, and similarly with $\eqcostaut:\mathbb{R}_{\ge 0} \rightarrow \mathbb{R}_{\ge 0}$ for the users of roads $\setautroads$. Both of these functions will be strictly increasing in their arguments as a result of our assumption that $\asymmetry_\idxroad > 0$ and $\scaleparam_\idxroad > 0$ $\forall \idxroad \in \setroads$. Similarly, the cost on the mixed road is increasing in both arguments.
	
	Further, note that the experienced cost of the human drivers and autonomous users on the mixed road will be the same on the mixed road, since they have identical tolls on that road. Formally,
	\begin{align}\label{eq:equal_cost_mixed}
		\costhum_\idxmixedrd(\flowhum_\idxmixedrd,\flowaut_\idxmixedrd) = \costaut_\idxmixedrd(\flowhum_\idxmixedrd,\flowaut_\idxmixedrd) \; .
	\end{align}
	Finally, note that that if there are human drivers on the mixed road, then $\costhum_\idxmixedrd(\flowhum_\idxmixedrd,\flowaut_\idxmixedrd) = \eqcosthum(\demandhum - \flowhum_\idxmixedrd)$ and similarly for autonomous vehicles. 
	
	Now, consider for the purpose of contradiction that there exists a second equilibrium $\alteqflow$ with greater social cost, \emph{i.e.} $\socialcost(\alteqhumflow,\alteqautflow)>\socialcost(\flowhumopt,\flowautopt)$. Due to the properties of essential uniqueness discussed above, to have a different social cost the new equilibrium must have different flow on the mixed road. We therefore first consider the case that $\alteqhumflow_\idxmixedrd > \flowhumopt_\idxmixedrd$. If $\alteqautflow_\idxmixedrd \ge \flowautopt_\idxmixedrd$ there is an immediate contradiction, as
	\begin{align*}
		\eqcosthum(\demandhum - \alteqhumflow_\idxmixedrd) < \eqcosthum(\demandhum - \flowhumopt_\idxmixedrd) = \costhum_\idxmixedrd(\flowhumopt_\idxmixedrd, \flowautopt_\idxmixedrd) < \costhum_\idxmixedrd(\alteqhumflow_\idxmixedrd, \alteqautflow_\idxmixedrd)
	\end{align*}
	violating the equilibrium conditions for the human-driven vehicles, contradicting the premise. If instead $\alteqhumflow_\idxmixedrd > \flowhumopt_\idxmixedrd$ and $\alteqautflow_\idxmixedrd < \flowautopt_\idxmixedrd$,
	\begin{align*}
		\eqcosthum(\demandhum - \alteqhumflow_\idxmixedrd) &< \eqcosthum(\demandhum - \flowhumopt_\idxmixedrd) = \costhum_\idxmixedrd(\flowhumopt_\idxmixedrd,\flowautopt_\idxmixedrd) = \costaut_\idxmixedrd(\flowhumopt_\idxmixedrd,\flowautopt_\idxmixedrd) \\
		& = \eqcostaut(\demandaut - \flowautopt_\idxmixedrd) < \eqcostaut(\demandaut - \flowautopt_\idxmixedrd) \; ,
	\end{align*}
	and the equilibrium conditions with \eqref{eq:equal_cost_mixed} yield
	\begin{align*}
		\eqcosthum(\demandhum - \alteqhumflow_\idxmixedrd) = \costhum_\idxmixedrd(\alteqhumflow_\idxmixedrd,\alteqautflow_\idxmixedrd) = \costaut_\idxmixedrd(\alteqhumflow_\idxmixedrd,\alteqautflow_\idxmixedrd) \; .
	\end{align*}
	Together this implies that $\costaut_\idxmixedrd(\alteqhumflow_\idxmixedrd,\alteqautflow_\idxmixedrd)< \eqcostaut(\demandaut - \flowautopt_\idxmixedrd)$ violating the conditions for equilibrium for $\alteqflow$. The same analysis can be done with switching autonomous and human-driven vehicles. The contradiction then proves that all equilibria will have the same flow on the mixed road, completing the proof.
\end{proof}

\begin{remark}
	We have a degree of freedom in choosing $\mu$. One would likely choose $\mu$ to be greater than the maximum latency in $\flowopt$ so as to avoid paying anyone to travel on a road.
\end{remark}
\begin{remark}
	We could relax the assumption that $\asymmetry_\idxroad=1$ for at most one road, since it would only matter if those roads were the mixed roads. If they are, then we have essential uniqueness there as well and a more complicated version of the proof would hold.
\end{remark}

\emph{Example. } We provide the following example to demonstrate the benefit of the tolling scheme described above. Consider the heterogeneous network in Fig.~\ref{fig:example_tolling}, which has three roads with varying characteristics. One can consider the top road to be a highway, on which autonomous vehicle platooning offers significant benefits, the middle road to be an urban road, and the bottom road to be a road in a residential community with many bike paths and pedestrian crossings. Let the human-driven flow demand be $\demandhum = 2.625$ and autonomous flow demand be $\demandaut = 2.5$.

\begin{figure}
	\centering
	\begin{tikzpicture}[->, >=stealth', auto, semithick, node distance=8cm]
	\tikzstyle{every state}=[fill=white,draw=black,thick,text=black,scale=0.8]
	\node[state]    (0)               {$s$};
	\node[state]    (1)[right of=0]   {$t$};
	\path
	(0) edge[bend left=50]		node{$\latency_1(\flowhum_1, \flowaut_1) = 4 \flowhum_1 + \flowaut_1 + 0.5$}     (1)
	(0) edge				node[above]{$\latency_2(\flowhum_2,\flowaut_2) = 2\flowhum_2 + \flowaut_2 + 1$}     (1)
	(0) edge[bend right=50]	node[below]{$\latency_3(\flowhum_3,\flowaut_3) = \flowhum_3 + 3\flowaut_3 + 0.5$}     (1);
	\end{tikzpicture}
	\caption{An example of a network that benefits from the tolling scheme described in this paper. Consider human-driven flow demand $\demandhum = 2.625$ and autonomous flow demand $\demandaut = 2.5$.}
	\label{fig:example_tolling}
\end{figure}
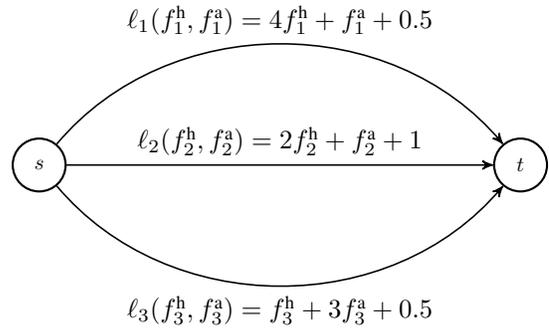

We list the numerical values for the worst-case equilibrium routing, the optimal routing, and the tolls that yield optimal routing in Table~\ref{table:example}. We choose our free variable for tolling to be $\mu=3$ to keep the tolls relatively low. In summary, in the worst-equilibrium routing, road $1$ has only human-driven vehicles and road $3$ has only autonomous vehicles; in the optimal routing this is reversed. As expected from Thm.~\ref{thm:separation}, in the optimal routing there is only one mixed road, which is road $2$. In the worst-case equilibrium all roads have positive flow and therefore have the same latency, which is $5$; the social cost is then $25.625$. In the optimal routing, which is enforced by the tolls provided, road latency varies but all latencies are well below $5$, and the social cost is $12.92$ -- an improvement by approximately a factor of $2$.

\begin{table}
	\caption{Routings, Latency, and Tolls for the Example Network}
	\label{table:example}
	\centering
	\begin{tabular}{lllllllll}
		\toprule
		 & \multicolumn{3}{c}{Worst eq. routing,} & \multicolumn{3}{c}{Opt. routing,} & \multicolumn{2}{c}{Opt. Tolls}  \\
		& \multicolumn{3}{c}{social cost: $25.625$ } & \multicolumn{3}{c}{social cost: $12.92$} &  \\
		\cmidrule(r){2-4}
		\cmidrule(r){5-7}
		\cmidrule(r){8-9}
		Road  & Human & Aut. & $\latency$ & Human & Aut. & $\latency$ & $\tollhum$ & $\tollaut$ \\
		\midrule
		1 & $1.125$ & $0$ & 5 & 0 & 1.65 & 1.67 & $\largenum$ & 1.33 \\
		2 & $1.5$ & $1$ & 5 & 0.37 & 0.85 & 2.58 & 0.42 & 0.42 \\
		3 & $0$ & $1.5$ & 5 & 2.26 & 0 & 2.76 & 0.24 & $ \largenum$\\
		\bottomrule
	\end{tabular}
\end{table}

\section{CONCLUSION AND FUTURE WORK}
\label{sct:conclusion}
In this paper we investigated tolling for roads with a mixture of human-driven and autonomous vehicles. We showed that if human drivers and autonomous vehicles are given the same tolls, the resulting equilibrium may have unboundedly worse total delay than the best-case routing. Allowing ourselves to toll human drivers and autonomous vehicles differently, we first established theoretical properties of the optimal routing of this mixed traffic on parallel roads, then used these results to find an optimal tolling scheme. 

There is room for expanding these results, specifically in the following directions.
\begin{itemize}
	\item This work deals with affine latency functions; the BPR model used generalizes easily to higher-order polynomials. It would be worthwhile to investigate tolling with these higher-order functions.
	\item This work could be expanded by analyzing more general network models, including those with multiple populations, each with its own source-destination pair.
	\item One can consider the Stackelberg case, assuming the routing for autonomous vehicles can be directly controlled. It may be fruitful to develop a unified Stackelberg (for autonomous vehicles) and tolling (for human-driven vehicles) scheme.
\end{itemize}
These would result in further steps towards ensuring the efficient operation of traffic networks with the emergence of autonomous vehicles.

\section{ACKNOWLEDGMENTS}
\label{sct:acknowledgments}
This work was supported by NSF grants CCF-1755808 and 1736582 and UC Office of the President grant LFR-18-548175. The authors also thank Mohit Srinivasan for many productive discussions.


\bibliographystyle{ieeetr}
\bibliography{refs.bib}

\end{document}